\documentclass[10pt]{amsart}

\usepackage{amssymb,amsmath,amsthm,amsfonts}
\usepackage{pdfsync}

\newcommand{\comment}[1]{}

\newtheorem{lem}{Lemma}[section]
\newtheorem{propn}{Proposition}[section]

\newtheorem{thm}{Theorem}[section]

\theoremstyle{remark}

\theoremstyle{definition}

\newcommand{\R}{\mathbb R}

\newcommand{\C}{\mathbb C}

\newcommand{\D}{\delta}

\newcommand{\VE}{\varepsilon}

\newcommand{\lm}{\lambda}

\newcommand{\be}{\begin{equation}}
\newcommand{\ee}{\end{equation}}
\newcommand{\bee}{\begin{equation*}}
\newcommand{\eee}{\end{equation*}}

\begin{document}
\title{Simplices and sets of positive upper density in $\R^d$ }
\author{Lauren Huckaba\quad\quad\quad Neil Lyall\quad\quad\quad\'Akos Magyar}
\address{Department of Mathematics, The University of Georgia, Athens, GA 30602, USA}
\email{lhuckaba@math.uga.edu}
\address{Department of Mathematics, The University of Georgia, Athens, GA 30602, USA}
\email{lyall@math.uga.edu}
\address{Department of Mathematics, The University of Georgia, Athens, GA 30602, USA}
\email{magyar@math.uga.edu}


\begin{abstract}
We prove an extension of Bourgain's theorem on pinned distances in measurable subset of $\mathbb{R}^2$ of positive upper density, namely Theorem $1^\prime$ in \cite{B}, to pinned non-degenerate $k$-dimensional simplices in measurable subset of $\mathbb{R}^{d}$ of positive upper density whenever $d\geq k+2$ and $k$ is any positive integer.
\end{abstract}
\maketitle

\setlength{\parskip}{1pt}


\section{Introduction}


Recall that the \emph{upper density} $\overline{\D}$ of a measurable set $A\subseteq\R^d$ is defined by \[\overline{\D}(A)=\limsup_{N\rightarrow\infty}\frac{|A\cap B_N|}{|B_N|},\]
where $|\cdot|$ denotes Lebesgue measure on $\R^d$ and 
$B_N$ denotes the cube $[-N/2,N/2]^d$.


\subsection{Existing results}

A result of Katznelson and Weiss \cite{FKW} states that if $A$ is a measurable subset of $\mathbb{R}^2$ of positive upper density, then its distance set
\[\text{dist}(A)=\{|x-y|\,:\, x,y\in A\}\] contains all large numbers. This result was later reproven using Fourier analytic techniques by Bourgain in \cite{B}.  Bourgain in fact established more, namely the following generalization and ``pinned variant''.

\begin{thm}[Theorem 2 in \cite{B}]\label{BourSimp}
Let $\Delta$ be a fixed non-degenerate $k$-dimensional simplex. If $A$ is a measurable subset of $\R^d$ of positive upper density with $d\geq k+1$, then
there exists a threshold $\lm_0=\lm_0(A,\Delta)$ 
such that for all $\lambda\geq \lm_0$ one has 
 \be
x+\lm\cdot U(\Delta)\subseteq A\ee
for some $x\in A$ and $U\in SO(d)$.
\end{thm}

In Theorem \ref{BourSimp}, and throughout this article, we refer to a set $\Delta=\{0,v_1,\dots,v_k\}$ of points in $\R^{k}$ as  a non-degenerate $k$-dimensional simplex if the vectors $\{v_1,\dots,v_k\}$ are linearly independent.

\begin{thm}[Pinned distances, Theorem $1^\prime$ in \cite{B}]\label{BourPinned}
If $A$ is a measurable subset of $\R^2$ of positive upper density, then there exist $\lm_0=\lm_0(A)$ such that for any given  $\lambda_1\geq \lm_0$ there is a fixed $x\in A$ such that
\be\label{opt3}
A\cap (x+\lm\cdot S^1)\ne\emptyset\ee
for all $\lm_0\leq\lm\leq\lm_1$, where $S^1$ denotes the unit circle centered at the origin in $\R^2$.
\end{thm}


\subsection{New results}


Throughout this article, we use $\mu$ to denote the unique Haar measure on $SO(d)$.

Our first result is the following (optimal) strengthening of Theorem \ref{BourSimp} above. 

\begin{thm}[Density of Embedded Simplices]\label{Thm1}
Let $\Delta=\{0,v_1,\dots,v_k\}\subseteq\R^k$ be a fixed non-degenerate $k$-dimensional simplex and $\VE>0$.

If $A$ is a measurable subset of $\R^d$ with $d\geq k+1$, then there exist $\lm_0=\lm_0(A,\Delta,\VE)$ such that
\be\label{DensityConclusionUD}
\int_{SO(d)}\overline{\D}(A\cap(A+\lm\cdot U(v_1))\cap\cdots \cap(A+\lm\cdot U(v_k)))\,d\mu(U)>\overline{\D}(A)^{k+1}-\VE\ee
for all $\lambda\geq \lm_0$.
In particular, for each $\lambda\geq \lm_0$ we may conclude that there exist $U\in SO(d)$ such that
\be\label{opt1}
\overline{\D}(A\cap(A+\lm\cdot U(v_1))\cap\cdots \cap(A+\lm\cdot U(v_k)))>\overline{\D}(A)^{k+1}-\VE\ee
and there exist $x\in A$ such that
\be\label{opt2}
\mu\bigl(\bigl\{U\in SO(d)\,:\, x+\lm\cdot U(\Delta)\subseteq A\bigr\}\bigr)>\overline{\D}(A)^{k}-2\VE.\ee
\end{thm}

The main result of this paper is the following (optimal) extension of Bourgain's pinned distances theorem, Theorem \ref{BourPinned} above, to non-degenerate $k$-dimensional simplices when $k\geq2$.

\begin{thm}[Density of Embedded Pinned Simplices]\label{Thm2}
Let $\Delta=\{0,v_1,\dots,v_k\}\subseteq\R^k$ be a fixed non-degenerate $k$-dimensional simplex and $\VE>0$.

If $A$ is a measurable subset of $\R^d$ with $d\geq k+2$, then there exist $\lm_0=\lm_0(A,\Delta,\VE)$ such that for any given  $\lambda_1\geq \lm_0$ there is a fixed $x\in A$ such that
\be\label{opt3}
\mu\bigl(\bigl\{U\in SO(d)\,:\, x+\lm\cdot U(\Delta)\subseteq A\bigr\}\bigr)>\overline{\D}(A)^{k}-\VE\quad\text{for all} \quad \lm_0\leq\lm\leq\lm_1.\ee
\end{thm}

We remark that Theorem \ref{Thm2} should hold whenever $d\geq k+1$. However, extending our result to this range appears to require an essentially non-Fourier analytic approach, specifically an adaptation of the geometric arguments in \emph{Bourgain's circular maximal function theorem} \cite{Bo_cir} to the configuration spaces considered in this article. We plan to address this strengthening of Theorem \ref{Thm2} in a separate article.

We further remark that both Theorem \ref{Thm1} and Theorem \ref{Thm2} also hold, with the proofs essentially unchanged, if the notion of upper density is replaced with the weaker notion of upper Banach density. 

\subsection{Outline of paper}

We will adapt Bourgain's approach in \cite {B} and deduce Theorems \ref{Thm1} and \ref{Thm2} from two quantitative compact variants, namely Propositions \ref{Thm1dic} and \ref{Thm2dic} respectively. The reduction of Theorems \ref{Thm1} and \ref{Thm2} to these ``dichotomy propositions" is carried out in Section \ref{Section2}. In Section \ref{prelim} we introduce a (natural) multi-linear averaging operator which we shall use to count the configuration under consideration as well as discuss some preliminary estimates before completing the proof of Proposition \ref{Thm1dic} in Section \ref{proof1}. In Section \ref{proof2} we reduce Proposition \ref{Thm2dic} to certain maximal function estimates over our configuration spaces, namely Propositions \ref{P1} and \ref{P2}, the proofs of which are presented in Section \ref{Maximal}.

\subsection{Further Notation}
Throughout this article we use the notation $X\ll Y$ to denote the fact that $X\leq C\, Y$ for some absolute constant $C>0$ that depends only the dimension $d$ and $X\lll Y$ to denote the fact that $X\leq c\, Y$ for some \emph{sufficiently small} constant $c>0$.

For any given set $A\subseteq\R^d$ we use $1_A$ to denote the characteristic function of the set $A$,
while for any given integrable function $f:\R^d\to\C$ we define its \emph{Fourier transform} $\widehat{f}:\R^d\to\C$, by
\be
\widehat{f}(\xi)=\int_{\R^d}f(x)e^{-2\pi i x\cdot \xi}\,dx.
\ee

\section{Reducing Theorems \ref{Thm1} and \ref{Thm2} to Key Dichotomy Propositions}

\subsection{Dichotomy Propositions}

In Section \ref{Section2} below we shall see that Theorems \ref{Thm1} and \ref{Thm2} are easy consequences of the following two quantitative compact variants, namely Propositions \ref{Thm1dic} and \ref{Thm2dic}.

Proposition \ref{Thm1dic} below (respectively Proposition \ref{Thm2dic}) establishes that if $A$ does not contain the ``expected" number of unpinned isometric copies of $\lm\cdot\Delta$ (respectively pinned isometric copies of $\lm\cdot\Delta$ with $\lm_0\leq\lm\leq \lm_1$ at some point $x\in A$), then this ``non-random" behavior will be ``detected" by the Fourier transform of the characteristic function of $A$ and result in a concentration of its $L^2$-mass on appropriate annuli.

\begin{propn}[Dichotomy for Theorem \ref{Thm1}]\label{Thm1dic}
Let $\Delta=\{0,v_1,\dots,v_k\}\subseteq\R^k$ be a fixed non-degenerate $k$-dimensional simplex, $\VE>0$, $0<\eta\lll\VE^{5/2}$, and $N\geq C_\Delta\eta^{-4}$.

If $A\subseteq B_N\subseteq \R^d$ with $d\geq k+1$, then for any $\lm$ satisfying $1\leq \lm\leq \eta^{4} N $ one of the following statements must hold:

 \begin{itemize}
\item[(i)] 
\[\int_{SO(d)}\frac{|A\cap(A+\lm\cdot U(v_1))\cap\cdots \cap(A+\lm\cdot U(v_k))|}{N^d}\,d\mu(U)>\left(\frac{|A|}{N^d}\right)^{k+1}-\VE\]
\vspace{-10pt}
\item[(ii)] \[\dfrac{1}{|A|}\int_{\Omega_{\lambda}} |\widehat{1_A}(\xi)|^2\,d\xi \gg \VE^2\]
where \[\Omega_{\lambda}=\Omega_{\lambda}(\eta)=\{\xi\in\mathbb{R}^d\,:\,\eta^2\,\lambda^{-1}\leq|\xi|\leq\eta^{-2}\lambda^{-1}\}.\]
\end{itemize}
\end{propn}

\begin{propn}[Dichotomy for Theorem \ref{Thm2}]\label{Thm2dic}
Let $\Delta=\{0,v_1,\dots,v_k\}\subseteq\R^k$ be a fixed non-degenerate $k$-dimensional simplex, $\VE>0$, $0<\eta\lll\VE^3$, and $N\geq C_\Delta\eta^{-4}$.

If $A\subseteq B_N\subseteq \R^d$ with $d\geq k+2$, then for any pair $(\lm_0,\lm_1)$ satisfying $1\leq\lambda_0\leq\lm_1\leq \eta^{4} N$ one of the following statements must hold:

 \begin{itemize}
\item[(i)] there exist $x\in A$  with the property that 
\[\mu\bigl(\bigl\{U\in SO(d)\,:\, x+\lm\cdot U(\Delta)\subseteq A\bigr\}\bigr)>\left(\frac{|A|}{N^d}\right)^{k}-\VE \quad \text{for all}\quad\lm_0\leq\lm\leq \lm_1\]
\vspace{-10pt}
\item[(ii)] \[\dfrac{1}{|A|}\int_{\Omega_{\lm_0,\lm_1}} |\widehat{1_A}(\xi)|^2\,d\xi \gg \VE^2\]
where \[\Omega_{\lm_0,\lm_1}=\Omega_{\lm_0,\lm_1}(\eta)=\{\xi\in\mathbb{R}^d\,:\,\eta^2\,\lm_1^{-1}\leq|\xi|\leq\eta^{-2}\lm_0^{-1}\}.\]
\end{itemize}
\end{propn}


\subsection{Proof of Theorems \ref{Thm1} and \ref{Thm2}}\label{Section2}


\subsubsection{Proof that Proposition \ref{Thm1dic} implies Theorem \ref{Thm1}}
Let $\VE>0$ and $0<\eta\ll\VE^{5/2}$.
Suppose that $A\subseteq \R^d$ with $d\geq k+1$ is a set for which the conclusion of Theorem \ref{Thm1} fails to hold, namely that there
exists arbitrarily large integers $\lm$ for which
\[\int_{SO(d)}\overline{\D}(A\cap(A+\lm\cdot U(v_1))\cap\cdots \cap(A+\lm\cdot U(v_k)))\,d\mu(U)\leq\overline{\D}(A)^{k+1}-\VE.\]

For a fixed integer $J\ggg\VE^{-2}$ we now choose a sequence $\{\lm^{(j)}\}_{j=1}^J$ of such $\lm$'s with the additional property that
$1\leq\lm^{(j)}\leq \eta^{4} \lm^{(j+1)}$
for $1\leq j<J$.
We now choose $N$ so that $\lm^{(J)}\leq\eta^4 N$ and that we simultaneously have both that
\be\label{densities}
\overline{\D}(A)^{k+1}-\VE/2\leq\left(\frac{|A\cap B_N|}{N^d}\right)^{k+1}-\VE/4
\ee
and that
\[\int_{SO(d)}\frac{|A_N\cap(A_N+\lm^{(j)}\cdot U(v_1))\cap\cdots \cap(A_N+\lm^{(j)}\cdot U(v_k))|}{N^d}\,d\mu(U)\leq\overline{\D}(A)^{k+1}-\VE/2
\]
holds for all $1\leq j\leq J$, where $A_N=A\cap B_N$. For the last inequality we exploited Fatou's Lemma. 

Abusing notation and denoting the set $A_N=A\cap B_N$ by $A$, an application of Proposition \ref{Thm1dic}, with $\VE$ replaced with $\VE/4$, thus allows us to conclude that for this set one must have
\be
\sum_{j=1}^J\frac{1}{|A|}\int_{\Omega_{\lm^{(j)}}} |\widehat{1_A}(\xi)|^2\,d\xi\gg J\VE^2 >1.
\ee
On the other hand it follows from the disjointness property of the sets $\Omega_{\lm^{(j)}}$, which we guaranteed by our initial choice of sequence $\{\lm^{(j)}\}$, and Plancherel's Theorem that
\be
\sum_{j=1}^J\frac{1}{|A|}\int_{\Omega_{\lm^{(j)}}} |\widehat{1_A}(\xi)|^2\,d\xi\leq \frac{1}{|A|}\int_{\R^d}|\widehat{1_A}(\xi)|^2\,d\xi= 1
\ee
giving a contradiction. \qed


\subsubsection{Proof that Proposition \ref{Thm2dic} implies Theorem \ref{Thm2}}

Let $\VE>0$ and $0<\eta\ll\VE^3$.
Suppose that $A\subseteq \R^d$ with $d\geq k+2$ is a set for which the conclusion of Theorem \ref{Thm2} fails to hold, namely that there
exists arbitrarily large pairs $(\lm_0,\lm_1)$ of real numbers such that for all $x\in A$ one has
\[\mu\bigl(\bigl\{U\in SO(d)\,:\, x+\lm\cdot U(\Delta)\subseteq A\bigr\}\bigr)\leq\overline{\D}(A)^{k}-\VE\]
for some $\lm_0\leq\lm\leq\lm_1$.

For a fixed integer $J\ggg\VE^{-2}$ we choose a sequence of such pairs $\{(\lm_0^{(j)},\lm_1^{(j)})\}_{j=1}^J$ with the property that
$1\leq\lm_1^{(j)}\leq \eta^{4} \lm_0^{(j+1)}$
for $1\leq j<J$.  
We now choose $N$ so that $\lm_1^{(J)}\leq\eta^{4} N$ and 
\be
\overline{\D}(A)^k-\VE\leq\left(\frac{|A\cap B_N|}{N^d}\right)^k-\VE/2.\ee

Abusing notation and denoting the set $A\cap B_N$ by $A$, an application of Proposition \ref{Thm2dic} thus allows us to conclude that for this set one must have
\be
\sum_{j=1}^J\frac{1}{|A|}\int_{\Omega_{\lm_0^{(j)},\lm_1^{(j)}}} |\widehat{1_A}(\xi)|^2\,d\xi\gg J\VE^2 >1.
\ee
On the other hand it follows from the disjointness property of the sets $\Omega_{\lm_0^{(j)},\lm_1^{(j)}}$, which we guaranteed by our initial choice of pair sequence $\{(\lm_0^{(j)},\lm_1^{(j)})\}$, and Plancherel's Theorem that
\be
\sum_{j=1}^J\frac{1}{|A|}\int_{\Omega_{\lm_0^{(j)},\lm_1^{(j)}}} |\widehat{1_A}(\xi)|^2\,d\xi\leq \frac{1}{|A|}\int_{\R^d}|\widehat{1_A}(\xi)|^2\,d\xi= 1
\ee
giving a contradiction. \qed

\section{Preliminaries}\label{prelim}

\subsection{The multi-linear operators $\mathcal{A}^{(j)}_{\lambda}$}
Let $\Delta=\{0,v_1,\dots,v_k\}$ be our fixed $k$-dimensional simplex. Without loss of generality we may assume that $|v_1|=1$.
For each $1\leq j\leq k$ we introduce the multi-linear operator
$\mathcal{A}^{(j)}_{\lambda}$, defined initially for Schwartz functions $g_{1},\dots,g_j$, by
\be\label{Aj}
\mathcal{A}^{(j)}_{\lambda}(g_1,\dots,g_j)(x)=\int\cdots\int g_{1}(x-\lm y_1)\cdots g_j(x-\lm y_j)\,d\sigma^{(d-j)}_{y_1,\dots,y_{j-1}}(y_j)\cdots d\sigma^{(d-1)}(y_{1})
\ee
where $\sigma^{(d-1)}$ denotes the normalized measure on the sphere $S^{d-1}(0,|v_1|)\subseteq\R^d$ induced by Lebesgue measure and
$\sigma^{(d-j)}_{y_1,\dots,y_{j-1}}$
denotes, for each $2\leq j\leq k$, the normalized measure on the spheres \be
S^{d-j}_{x_1,\dots,x_{j-1}}=S^{d-1}(0,|v_j|)\cap S^{d-1}(x_1,|v_j-v_1|)\cap\cdots\cap S^{d-1}(x_{j-1},|v_j-v_{j-1}|)\ee
where $S^{d-1}(x,r)=\{x'\in\R^d\,:\,|x-x'|=r\}$.

\comment{
where $d\sigma^{(d-1)}$ denotes the measure on the unit sphere $S^{d-1}\subseteq\R^d$ induced by Lebesgue measure normalized to have total mass $1$ and
$d\sigma^{(d-j)}_{y_1,\dots,y_{j-1}}$
denotes, for each $2\leq j\leq k$, the normalized measure on the sphere \[S^{d-j}_{y_1,\dots,y_{j-1}}\subseteq 
y+[y_1,\dots,y_{j-1}]^\perp\simeq\R^{d-j+1}\]
of radius $r_j=\text{dist}(v_j,[v_1,\dots,v_{j-1}])$ centered at $y\in[y_1,\dots,y_{j-1}]$ with $y\cdot y_i=v_j\cdot v_i$
for all $1\leq i\leq j-1$.}

The multi-linear operator $\mathcal{A}^{(j)}_{\lambda}$ is a natural object for us to consider in light of the observation that it could have equivalently be defined for each $1\leq j\leq k$ using the formula
\be
\mathcal{A}^{(j)}_{\lambda}(g_1,\dots,g_j)(x):=\int_{SO(d)}g_1(x-\lm\cdot U(v_1))\cdots g_j(x-\lm\cdot U(v_{j}))\,d\mu(U)
\ee
and hence for any bounded measurable set $A\subseteq\R^d$, the quantity
\be
\bigl\langle 1_A,\mathcal{A}^{(k)}_{\lambda}(1_A,\dots,1_A)\bigr\rangle=\int_{SO(d)}|A\cap(A+\lm\cdot U(v_1))\cap\cdots \cap(A+\lm\cdot U(v_k))|\,d\mu(U).\ee

A trivial, but important, observation will be the fact that
\be\label{jtoj-1}
\Bigl| \mathcal{A}^{(j)}_{\lambda}(g_1,\dots,g_j)(x)-g_j(x)\,\mathcal{A}^{(j-1)}_{\lambda}(g_1,\dots,g_{j-1})(x)\Bigr|\leq
  \int \bigl| g_j(x-\lm y)-g_j(x)\bigr|\,d\sigma^{(d-j)}_{y_1,\dots,y_{j-1}}(y)
\ee
holds for some initial choices of frame $y_1,\dots,y_{j-1}$ with $y_i\cdot y_{i'}=v_i\cdot v_{i'}$ for $1\leq i\leq i'\leq j-1$.

\subsection{A second averaging operator and some basic estimates}

We now introduce a second averaging operator, which we also denote by $\mathcal{A}^{(j)}_{\lambda}$, defined initially for any Schwartz function $g$, by
\be
\mathcal{A}^{(j)}_{\lm}(g)(x)=\int\cdots\int\Bigl|  \int g(x-\lm y_j)\,d\sigma^{(d-j)}_{y_1,\dots,y_{j-1}}(y_j) \Bigr|\,
d\sigma^{(d-j+1)}_{y_1,\dots,y_{j-2}}(y_{j-1})\cdots d\sigma^{(d-1)}(y_{1})
\ee

Note that if the functions $g_1,\dots,g_{j-1}$ are all bounded in absolute value by $1$, then clearly
\be\label{2Aj}
\bigl| \mathcal{A}^{(j)}_{\lambda}(g_1,\dots,g_j)(x)\bigr|\leq
\mathcal{A}^{(j)}_{\lambda}(g_j)(x).
\ee

Fix $1\leq j\leq k$. It is easy to see, using Minkowski's inequality, that for any Schwartz function $g$ we have the crude estimate
\be\label{trivial}
\int \bigl| \mathcal{A}^{(j)}_{\lambda}(g)(x)\bigr|^2\,dx\leq 
\int |g(x)|^2\,dx.
\ee

However, arguing more carefully one can just as easily obtain, using Plancherel's identity, the estimate 
\be\label{transformside}
\int \bigl| \mathcal{A}^{(j)}_{\lambda}(g)(x)\bigr|^2\,dx\leq
\int\cdots\int\Bigl(
\int |\widehat{g}(\xi)|^2\bigl|\widehat{d\sigma^{(d-j)}_{y_1,\dots,y_{j-1}}}(\lm\,\xi)\bigr|^2\,d\xi
 \Bigr)\,
d\sigma^{(d-j+1)}_{y_1,\dots,y_{j-2}}(y_{j-1})\cdots d\sigma^{(d-1)}(y_{1}),
\ee
where as usual 
\be
\widehat{d\nu}(\xi)=\int_{\R^d}e^{-2\pi i x\cdot\xi}\,d\nu(x)\ee
denotes the Fourier transform of any complex-valued Borel measure $d\nu$.
In light of (\ref{transformside}) it will come as little surprise that is the course of our arguments we will have use for the basic estimate
\be\label{decay}
\bigl|\widehat{d\sigma^{(d-j)}_{y_1,\dots,y_{j-1}}}(\xi)\bigr|+\bigl|\nabla\widehat{d\sigma^{(d-j)}_{y_1,\dots,y_{j-1}}}(\xi)\bigr|\leq C_\Delta\bigl(1+\text{dist}(\xi,\text{span}\{y_1,\dots,y_{j-1}\})\bigr)^{-(d-j)/2},
\ee
which is a consequence of the well-known asymptotic behavior of the Fourier transform of the measure on the unit sphere $S^{d-j}\subseteq\R^{d-j+1}$ induced by Lebesgue measure, see for example \cite{Stein}.


\subsection{A smooth cutoff function $\psi$ and some basic properties
}\label{properties}

Let
 $\psi:\R^d\rightarrow(0,\infty)$ be a Schwartz function that satisfies
\[1=\widehat{\psi}(0)\geq\widehat{\psi}(\xi)\geq0\quad\quad\text{and}\quad\quad \widehat{\psi}(\xi)=0 \ \ \text{for} \ \ |\xi|>1.\]
As usual, for any given $t>0$, we define
\be
\psi_{t}(x)=t^{-d}\psi(t^{-1}x).\ee

First we record the trivial observation that
\[\int\psi_{t}(x)\,dx=\int\psi(x)\,dx=\widehat{\psi}(0)=1\]
as well as the simple, but important, observation that $\psi$ may be chosen so that 
\be\label{cutoff2}
\bigl| 1-\widehat{\psi}_{t}(\xi)\bigr|=\bigl| 1-\widehat{\psi}(t\xi)\bigr|\ll\min\{1,t|\xi|\}.
\ee

Finally we record a formulation, appropriate to our needs, of the fact that for any given small parameter $\eta$, our cutoff function $\psi_{t}(x)$ will essentially supported where $|x|\leq\eta^{-1}t$ and is approximately constant on smaller scales. More precisely,

\begin{lem} Let $\eta>0$ and $t>0$, then 
\be\label{5.3}
\int_{|x|\geq\eta^{-1}t}\psi_{t}(x)\,dx \ll \eta.
\ee
and 
\be\label{5.1}
\int\int\bigl|\psi_{t}(x-\lm y)-\psi_{t}(x)\bigr|\,d\sigma^{(d-j)}_{y_1,\dots,y_{j-1}}(y)\,dx \ll \eta
\ee
for any $1\leq j\leq k$ and any choice of frame $y_1,\dots,y_{j-1}$ provided  $t\geq \eta^{-1}\lm$.
\end{lem}

\begin{proof} 
Estimate (\ref{5.3}) is easily verified using the fact that $\psi$ is a Schwartz function on $\R^d$ as
\[\int_{|x|\geq\eta^{-1}t}\psi_{t}(x)\,dx= \int_{|x|\geq\eta^{-1}}\psi(x)\,dx
\ll \int_{|x|\geq\eta^{-1}}(1+|x|)^{-d-1}\,dx\ll\, \eta.
\]

To verify estimate (\ref{5.1}) we make use of the fact that both
 $\psi$ and its derivative are rapidly decreasing, specifically
\begin{align*}
\int\int\bigl|\psi_{t}(x-\lm y)-\psi_{t}(x)\bigr|\,d\sigma^{(d-j)}_{y_1,\dots,y_{j-1}}(y)\,dx
&\leq \int\int\bigl|\psi(x-\lm y/t)-\psi(x)\bigr|\,d\sigma^{(d-j)}_{y_1,\dots,y_{j-1}}(y)\,dx\\
&
\ll \frac{\lm}{t} \int(1+|x|)^{-d-1}\,dx\ll \frac{\lm}{t}\end{align*}
for any choice of frame $y_1,\dots,y_{j-1}$.\end{proof}


\section{Proof of Proposition \ref{Thm1dic}}\label{proof1}

Let $f=1_A$ and $\D=|A|/N^d$.
Suppose that $1\leq\lm\leq \eta^{4} N$ and that (i) does not hold, then
\begin{equation}\label{not1U}
\langle f,\mathcal{A}^{(k)}_{\lambda}(f,\dots,f)\rangle\leq\langle f,\D^k-\VE\rangle=(\D^k-\VE)|A|.
\end{equation}

If we let $f_1:=f*\psi_{\eta^{-1}\lm}$, then by (\ref{jtoj-1}) and  (\ref{5.1}) it follows that for all $x\in\R^d$ and $1\leq j\leq k$ we have
\be
\Bigl|  \mathcal{A}^{(j)}_{\lambda}(f,\dots,f,f_1)(x)-f_1(x)\,\mathcal{A}^{(j-1)}_{\lambda}(f,\dots,f)(x)      \Bigr|\ll \eta
\ee
and
consequently
\be
f_1(x)^k+\sum_{j=1}^k f_1(x)^{k-j}\mathcal{A}^{(j)}_{\lambda}(f,\dots,f,f-f_1)(x)\ll\mathcal{A}^{(k)}_{\lambda}(f,\dots,f)(x)+\eta.
\ee
Together  with (\ref{not1U}) this gives
\be\label{allj}
\sum_{j=1}^k \bigl\langle ff_1^{k-j},\mathcal{A}^{(j)}_{\lambda}(f,\dots,f,f-f_1)\bigr\rangle
\leq 
\langle f,\D^k-f_1^k-\VE/2\rangle
\ee
provided $\eta\lll\VE$. We will now combine this with the following result, which we isolate as a lemma.

\begin{lem}\label{f1k}
Let $0<\eta\ll\D$ and $f_1:=f*\psi_{\eta^{-1}\lm}$, then
\be
\langle f,\D^k-f_1^k\rangle\ll \eta|A|
\ee
\end{lem}
Combining Lemma \ref{f1k} with (\ref{allj}) we see that if $\eta\lll\VE$ and (\ref{not1U}) holds, 
then there exist $1\leq j\leq k$ such that 
\be
\Bigl| \bigl\langle ff_1^{k-j},\mathcal{A}^{(j)}_{\lambda}(f,\dots,f,f-f_1)\bigr\rangle  \Bigr|
\gg \VE |A|
\ee
and hence, using (\ref{2Aj}) and the fact that $0\leq f_1\leq 1$, that
\be\label{MTU}
\bigl\langle f,\mathcal{A}^{(j)}_{\lambda}(f-f_1)\bigr\rangle \gg \VE |A|.
\ee

The final ingredient in the proof of Proposition \ref{Thm1dic} is the following


\begin{lem}[Error term]\label{ETU}
If $f_2:=f*\psi_{\eta^2\lm}$, 
then for any $1\leq j\leq k$ we have the estimate
\be
\bigl\langle f,\mathcal{A}^{(j)}_{\lambda}(f-f_2)\bigr\rangle
\ll \eta^{2/5}|A|.\ee
\end{lem}


Indeed, since
\bee
 \bigl\langle f,\mathcal{A}^{(j)}_{\lambda}(f_2-f_1)\bigr\rangle 
\geq
 \bigl\langle f,\mathcal{A}^{(j)}_{\lambda}(f-f_1)\bigr\rangle  
-
 \bigl\langle f,\mathcal{A}^{(j)}_{\lambda}(f-f_2)\bigr\rangle   
\eee
we see that (\ref{MTU}) together with Lemma \ref{ETU} will imply that if $\eta\lll\VE^{5/2}$ and (\ref{not1U}) holds, then there exist $1\leq j\leq k$ such that
\be
 \bigl\langle f,\mathcal{A}^{(j)}_{\lambda}(f_2-f_1)\bigr\rangle  \gg \VE |A|.
\ee
It then follows, via Cauchy-Schwarz and Plancherel, that 
\be
\int \bigl|\widehat{f}(\xi)\bigr|^2\bigl| \widehat{\psi}_{\eta^2\lm}(\xi)-\widehat{\psi}_{\eta^{-1}\lm}(\xi)\bigr|^2\,d\xi\gg\VE^2\,|A|,
\ee
which is essentially the estimate that we are trying to prove and since (\ref{cutoff2}) implies that
\be\label{cutoff}
\bigl| \widehat{\psi}_{\eta^2\lm}(\xi)-\widehat{\psi}_{\eta^{-1}\lm}(\xi)\bigr|\ll\eta\ee
whenever $\xi\notin\Omega_{\lm}$, it indeed sufficies and concludes the proof of Proposition \ref{Thm1dic}.\qed




\subsection{Proof of Lemma \ref{f1k}}
It suffices to establish the result when $k=1$, specifically that
\be\label{k=1}
\int f(x)f_1(x)\,dx\geq \D(1-C\eta)\,|A|
\ee
for some constant $C>0$, since from H\"older's inequality we would then obtain
\[\D^k(1-C\eta)^k\,|A|^k\leq\Bigl(\int f(x)f_1(x)\,dx\Bigr)^k\leq|A|^{k-1}\int f(x)f_1(x)^k\,dx
\]
from which the full result immediately follows since $0<\eta\lll1$. Towards establishing (\ref{k=1}) we note that using Parserval and the fact that $0\leq\widehat{\psi}\leq 1$ we have
\be
\int f(x)f_1(x)\,dx=\int |\widehat{f}(\xi)|^2\widehat{\psi}(\eta^{-1}\lm\,\xi)\,d\xi\geq \int |\widehat{f}(\xi)|^2|\widehat{\psi}(\eta^{-1}\lm\,\xi)|^2\,d\xi=\int f_1(x)^2\,dx
\ee
and as such we need only show that
\be\label{finally}
\int f_1(x)^2\,dx\geq \D(1-C\eta)\,|A|
\ee
for some constant $C>0$.
Since an application of Cauchy-Schwarz gives that
\be
\int_{B_N} f_1(x)^2\,dx\geq \frac{1}{|B_N|}\Bigl(\int_{B_N} f_1(x)\,dx\Bigr)^2
\ee
our task is further reduces to simply showing that
\be\label{w}
\int_{B_N} f_1(x)\,dx\geq (1-C\eta)|A|
\ee
for some constant $C>0$.
To establish (\ref{w}) we now let $N'=N+\eta^{-2} \lm$ and write 
\be\int_{\R^d} f_1(x)\,dx=\int_{B_N} f_1(x)\,dx+\int_{\R^d\setminus B_{N'}} f_1(x)\,dx+\int_{B_{N'}\setminus B_N} f_1(x)\,dx.\ee
The fact that $\lm\leq\eta^{4}N$ ensures that
 \be\frac{|B_{N'}\backslash B_N|}{|B_N|}\ll\left(\frac{N'}{N}-1\right)\ll \eta^{-2}\frac{\lm}{N}\ll\eta^{2}\ee
and hence, since $\eta\lll\D$ and $0\leq f_1\leq 1$, that
\[\int_{B_{N'}\setminus B_N} f_1(x)\,dx\ll\eta^2|B_N|\leq \eta|A|,\]
while  (\ref{5.3}) ensures that
\be\int_{\R^d\setminus B_{N'}} f_1(x)\,dx \leq |A|\int_{|y|\gg \eta^{-2}\lm} \psi_{\eta^{-1}\lm}(y)\,dy \ll \eta |A|.\ee
Since
\[\int_{\R^d} f_1(x)=\int_{\R^d} f(x)=|A|\]
estimate (\ref{w}) follows.\qed

\subsection{Proof of Lemma \ref{ETU}}

It follows from an application of Cauchy-Schwarz and Plancherel that
\[ \bigl\langle f,\mathcal{A}^{(j)}_{\lambda}(f-f_2)\bigr\rangle^2 \leq 
|A|\cdot\int|\widehat{f}(\xi)|^2|1-\widehat{\psi}(\eta^2\lm\,\xi)|^2I(\lm\,\xi)\,d\xi\]
where
\be\label{I}
I(\xi)=\int\cdots\int
\bigl|\widehat{d\sigma^{(d-j)}_{y_1,\dots,y_{j-1}}}(\xi)\bigr|^2\,d\sigma^{(d-j+1)}_{y_1,\dots,y_{j-2}}(y_{j-1})\cdots d\sigma^{(d-1)}(y_{1}).
\ee

While from (\ref{decay}), the trivial uniform bound $I(\xi)\ll 1$, and an appropriate ``conical" decomposition, depending on $\xi$, of the configuration space over which the integral $I(\xi)$ is defined, we have
\be\label{Ibound}
I(\xi)\leq C_\Delta (1+|\xi|)^{-(d-j)/2}.
\ee

Combining this observation with (\ref{cutoff2}) we obtain the uniform bound
\be
|1-\widehat{\psi}(\eta^2\lm\,\xi)|^2I(\lm\,\xi)\ll\min\{(\lm|\xi|)^{-1/2},\eta^4\lm^2|\xi|^2\}\leq\eta^{4/5}
\ee
which, after an application of Plancherel, completes the proof.
\qed





\section{Proof of Proposition \ref{Thm2dic}}\label{proof2}

Suppose that we have a pair
$(\lm_0,\lm_1)$ satisfying $1\leq\lambda_0\leq\lm_1\leq \eta^{4} N$, but for which (i) does not hold. It follows that for all $x\in A$ there must exist $\lm_0\leq\lm\leq\lm_1$ such that
\begin{equation}\label{not1Uagain}
\mathcal{A}^{(k)}_{\lambda}(f,\dots,f)(x)\leq\D^k-\VE.
\end{equation}

We now let $f_1=f*\psi_{\eta^{-1}\lm_1}$, noting the slight difference from the definition of $f_1$ given in the proof of Proposition \ref{Thm1dic}. It follows from (\ref{not1Uagain}), as in the proof of Proposition \ref{Thm1dic}, that for all $x\in A$ there must exist $\lm_0\leq\lm\leq\lm_1$ such that
\begin{equation}
\sum_{j=1}^k f_1(x)^{k-j}\mathcal{A}^{(j)}_{\lambda}(f,\dots,f,f-f_1)(x)
\leq  \D^k-f_1(x)^k-\VE/2
\end{equation}
provided $\eta\lll\VE$, and hence that 
\begin{equation}
\sum_{j=1}^k \mathcal{A}^{(j)}_{\,*}(f-f_1)(x)
\geq  f_1(x)^k-\D^k+\VE/2
\end{equation}
for all $x\in A$, where for any Schwartz function $g$, $\mathcal{A}_{\,*}^{(j)}(g)$ denotes the \emph{maximal average} defined by
\be
\mathcal{A}_{\,*}^{(j)}(g)(x):=\sup_{\lm_0\leq\lm\leq\lm_1}\mathcal{A}^{(j)}_{\lambda}(g) (x).\ee

 Consequently, provided $\eta\ll\VE$ and appealing to Lemma \ref{f1k}, we may conclude that there must exist $1\leq j\leq k$ such that 
\be\label{not1P}
 \bigl\langle f,\mathcal{A}^{(j)}_{\,*}(f-f_1)\bigr\rangle  
\gg \VE |A|.
\ee

Arguing as in the proof of Proposition \ref{Thm1dic} we see that everything reduces to establishing the $L^2$-boundedness of $\mathcal{A}_{\,*}^{(j)}$ together with appropriate estimates for the ``mollified'' maximal operator
\be
\mathcal{M}_{\eta}^{(j)}(f):=\mathcal{A}_{\,*}^{(j)}(f-f_2)
\ee
where $f_2=f*\psi_{\eta^2\lm_0}$. 

Note that
\be
\mathcal{M}_{\eta}^{(j)}(f)=\sup_{\lm_0\leq\lm\leq\lm_1}\int\cdots\int\Bigl|  \int f(x-\lm y_j)\,d\mu^{(j)}_{\eta}(y_j) \Bigr|\,
d\sigma^{(d-j+1)}_{y_1,\dots,y_{j-2}}(y_{j-1})\cdots d\sigma^{(d-1)}(y_{1})\ee
where
\be
d\mu^{(j)}_{\eta}=d\sigma^{(d-j)}_{y_1,\dots,y_{j-1}}-\psi_{\eta^2\lm_0\lm^{-1}}*d\sigma^{(d-j)}_{y_1,\dots,y_{j-1}}.
\ee
and hence 
\be
\widehat{\mu^{(j)}_{\eta}}(\lm\,\xi)=\widehat{d\sigma^{(d-j)}_{y_1,\dots,y_{j-1}}}(\lm\,\xi)\,\bigl(
1-\widehat{\psi}(\eta^2\lm_0\,\xi)
\bigr).
\ee

The precise results that we need are recorded in the following two propositions.

\begin{propn}[$L^2$-Boundedness of the Maximal Averages  $\mathcal{A}_{\,*}^{(j)}$]\label{P1}
If $d\geq j+2$, then
\be
\int_{\R^d}|\mathcal{A}_{\,*}^{(j)}(g)(x)|^2\,dx\ll \int_{\R^d}|g(x)|^2\,dx.\ee
\end{propn}

\begin{propn}[$L^2$-decay of the ``Mollified" Maximal Averages  $\mathcal{M}_{\eta}^{(j)}$]\label{P2}
Let $\eta>0$. 
If $d\geq j+2$, then
\be
\int_{\R^d}|\mathcal{M}_{\eta}^{(j)}(f)(x)|^2\,dx\ll\eta^{2/3} \int_{\R^d}|f(x)|^2\,dx.
\ee
\end{propn}

The proofs of Propositions \ref{P1} and \ref{P2} are presented in Section \ref{Maximal} below.
\qed


\section{Proof of Propositions \ref{P1} and \ref{P2}}\label{Maximal}

\subsection{Proof of Propositions \ref{P1}}
We first note that 
Cauchy-Schwarz ensures 
\[
\int_{\R^d}|\mathcal{A}_{\,*}^{(j)}(g)(x)|^2\,dx
\leq 
\int\cdots\int \int_{\R^d}\sup_{\lm_0\leq\lm\leq\lm_1}\Bigl|  \int g(x-\lm y_j)\,d\sigma^{(d-j)}_{y_1,\dots,y_{j-1}}(y_j) \Bigr|^2\,dx\,
d\sigma^{(d-j+1)}_{y_1,\dots,y_{j-2}}(y_{j-1})\cdots d\sigma^{(d-1)}(y_{1}).
\]

Now for fixed $y_1,\dots,y_{j-1}$ we can clearly identify $[y_1,\dots,y_{j-1}]^\perp$ with $\R^{d-j+1}$ and $d\sigma^{(d-j)}_{y_1,\dots,y_{j-1}}$ with a constant (depending only on $d$ and $\Delta$) multiple of $d\sigma^{(d-j)}$, the normalized measure on the unit sphere $S^{d-j}\subseteq\R^{d-j+1}$ induced by Lebesgue measure. Writing $\R^d=\R^{j-1}\times\R^{d-j+1}$, $g(x)=g_{x'}(x'')$, and applying \emph{Stein's spherical maximal function theorem} for functions in $L^2(\R^{d-j+1})$, see Section 5.5 in \cite{Graf}, which asserts that  
\be
\int_{\R^{d-j+1}}\sup_{\lm_0\leq\lm\leq\lm_1}\Bigl|  \int g(x-\lm y)\,d\sigma^{(d-j)}(y) \Bigr|^2\,dx
\ll \int_{\R^{d-j+1}}|g(x)|^2\,dx
\ee
whenever $d\geq j+2$, gives
\begin{align*}
\int_{\R^d}\sup_{\lm_0\leq\lm\leq\lm_1}\Bigl|  \int g(x-\lm y)&\,d\sigma^{(d-j)}_{y_1,\dots,y_{j-1}}(y) \Bigr|^2\,dx\\
&= C_\Delta \int_{\R^{j-1}}\int_{\R^{d-j+1}}\sup_{\lm_0\leq\lm\leq\lm_1}\Bigl|  \int g_{x'}(x''-\lm y)\,d\sigma^{(d-j)}(y) \Bigr|^2\,dx''\,dx'\\
&\leq C \int_{\R^{j-1}}\int_{\R^{d-j+1}}|g_{x'}(x'')|^2\,dx''\,dx'=C\int_{\R^d}|g(x)|^2\,dx
\end{align*}
with the constant $C$ independent of the initial choice of frame $y_1,\dots,y_{j-1}$.
The result follows.
\qed

\subsection{Proof of Propositions \ref{P2}}
We will deduce the validity of Proposition \ref{P2} from the following result for the slightly more general class of operators defined for any $L>0$ by
\be
\mathcal{M}_{L}^{(j)}(f)=\sup_{\lm_0\leq\lm\leq\lm_1}\int\cdots\int\Bigl|  \int f(x-\lm y)\,d\mu^{(j)}_{L}(y) \Bigr|\,
d\sigma^{(d-j+1)}_{y_1,\dots,y_{j-2}}(y_{j-1})\cdots d\sigma^{(d-1)}(y_{1})\ee
where 
\be
\widehat{d\mu^{(j)}_{L}}(\lm\,\xi)=m_L(\xi)\,\widehat{d\sigma^{(d-j)}_{y_1,\dots,y_{j-1}}}(\lm\,\xi)
\ee
with the multiplier $m_L$ now any smooth function that satisfies the estimate
\be\label{new}
|m_L(\xi)|\ll\min\{1,L|\xi|\}.
\ee

Recall that estimate (\ref{cutoff2}) is precisely the statement that $|1-\widehat{\psi}(L\xi)|\ll\min\{1,L|\xi|\}.$

\begin{thm}\label{MollifiedTheorem} 
If $d\geq j+2$ and $0<L\leq\lm_0$, then
\be
\int_{\R^d}|\mathcal{M}_{L}^{(j)}(f)(x)|^2\,dx\ll\Bigl(\frac{L}{\lm_0}\Bigr)^{1/3} \int_{\R^d}|f(x)|^2\,dx.
\ee
\end{thm}
\begin{proof}
An application of Cauchy-Schwarz gives
\be
\int_{\R^d}|\mathcal{M}_{L}^{(j)}(f)(x)|^2\,dx
\leq 
\int\cdots\int \left[\,\int_{\R^d}\sup_{\lm_0\leq\lm\leq\lm_1}|  M_{L,\lm} (f)(x)|^2\,dx\,\right]\,
d\sigma^{(d-j+1)}_{y_1,\dots,y_{j-2}}(y_{j-1})\cdots d\sigma^{(d-1)}(y_{1}).
\ee
where $M_{L,\lm}$ is the Fourier multiplier operator defined by
\be
\widehat{M_{L,\lm}(f)}(\xi)=\widehat{f}(\xi)\,m_L(\xi)\,\widehat{d\sigma^{(d-j)}_{y_1,\dots,y_{j-1}}}(\lm\,\xi).
\ee

A standard application of the Fundamental Theorem of Calculus, see for example \cite{Graf}, gives
\be
\sup_{\lm_0\leq\lm\leq\lm_1}|M_{L,\lm} (f)(x)|^2\leq 
2\int_{\lm_0}^{\lm_1}|M_{L,t} (f)(x)||\widetilde{M}_{L,t}(f)(x)|\,\frac{dt}{t}+|M_{L,\lm_0} (f)(x)|^2
\ee
where $\widetilde{M}_{L,t}(f)=t\dfrac{d}{dt}M_{L,t}(f)$. We further note that $\widetilde{M}_{L,t}$ is clearly also a Fourier multiplier operator, indeed
\be
\widehat{\widetilde{M}_{L,t}(f)}(\xi)=\widehat{f}(\xi)\,m_L(\xi)\,\bigl(t\xi\cdot\nabla \widehat{d\sigma^{(d-j)}_{y_1,\dots,y_{j-1}}}(t\xi)\bigr).
\ee

We now immediately see that
\begin{align*}
\int_{\R^d}|\mathcal{M}_{L}^{(j)}&(f)(x)|^2\,dx\\
&\leq
2\!\!\!\!\sum_{\ell=\lfloor\log_2\lm_0\rfloor}^\infty \int_{2^{\ell-1}}^{2^\ell}
\int\cdots\int \int_{\R^d}
|M_{L,t} (f)(x)||\widetilde{M}_{L,t} (f)(x)|\,dx\,d\sigma^{(d-j+1)}_{y_1,\dots,y_{j-2}}(y_{j-1})\cdots d\sigma^{(d-1)}(y_{1})\,\frac{dt}{t}\\
&\quad\quad\quad\quad \quad\quad\quad\quad 
+\int\cdots\int \int_{\R^d}
|M_{L,\lm_0} (f)(x)|^2\,dx\,d\sigma^{(d-j+1)}_{y_1,\dots,y_{j-2}}(y_{j-1})\cdots d\sigma^{(d-1)}(y_{1}).
\end{align*}

Applying Cauchy-Schwarz to the first integral above (in the variables $x, y_1,\dots,y_{j-1}$, and $t$ together), followed by an application of Plancherel (in two resulting integrations in $x$ as well as in the one that appears in the second integral above), we obtain the estimate
\be\label{111}
\int_{\R^d}|\mathcal{M}_{L}^{(j)}(f)(x)|^2\,dx
\leq
2\!\!\!\!\sum_{\ell=\lfloor\log_2\lm_0\rfloor}^\infty \bigl(\,\mathcal{I_\ell}\,\mathcal{\widetilde{I}_\ell}\,\bigr)^{1/2} + \mathcal{I}
\ee
with
\be
\mathcal{I_\ell}=\int_{2^{\ell-1}}^{2^\ell}\int_{\R^d}|\widehat{f}(\xi)|^2|m_L(\xi)|^2 
I(t\,\xi)\,d\xi\,\frac{dt}{t}
\ee
\be
\mathcal{\widetilde{I}_\ell}=\int_{2^{\ell-1}}^{2^\ell}\int_{\R^d}|\widehat{f}(\xi)|^2|m_L(\xi)|^2 
\widetilde{I}(t\,\xi)\,d\xi\,\frac{dt}{t}
\ee
and
\be
\mathcal{I}=\int_{\R^d}|\widehat{f}(\xi)|^2|m_L(\xi)|^2 
I(\lm_0\,\xi)\,d\xi
\ee
where, as in the proof of Proposition \ref{ETU}, we have defined
\be
I(\xi)=\int\cdots\int
\bigl|\widehat{d\sigma^{(d-j)}_{y_1,\dots,y_{j-1}}}(\xi)\bigr|^2\,d\sigma^{(d-j+1)}_{y_1,\dots,y_{j-2}}(y_{j-1})\cdots d\sigma^{(d-1)}(y_{1})
\ee
and analogously now also define
\be
\widetilde{I}(\xi)=\int\cdots\int
\bigl|\xi\cdot\nabla \widehat{d\sigma^{(d-j)}_{y_1,\dots,y_{j-1}}}(\xi)\bigr|^2\,d\sigma^{(d-j+1)}_{y_1,\dots,y_{j-2}}(y_{j-1})\cdots d\sigma^{(d-1)}(y_{1}).
\ee

Combining (\ref{new}) with (\ref{Ibound}), and recalling that we are assuming that $d\geq j+2$, gives
\be
|m_L(\xi)|^2I(t\,\xi)\ll\min\{(t|\xi|)^{-1},L^2|\xi|^2\}\leq L^{2/3}t^{-2/3}
\ee
which ensures, via Plancherel, that
\be\label{222}
\mathcal{I_\ell}\ll \Bigl(\frac{L}{2^\ell}\Bigr)^{2/3}\|f\|_2^2 \quad\text{and}\quad \mathcal{I}\ll \Bigl(\frac{L}{\lm_0}\Bigr)^{2/3}\|f\|_2^2.
\ee

Arguing as in the proof of estimate (\ref{Ibound}), we can see that estimate (\ref{decay}) for $\nabla \widehat{d\sigma^{(d-j)}_{y_1,\dots,y_{j-1}}}(\xi)$ ensures that
$\widetilde{I}(\xi)$ is bounded whenever $d\geq j+2$, note also that it is not bounded if $d=k+1$. It follows immediately from this observation (and Plancherel) that
\be\label{333}
\mathcal{\widetilde{I}_\ell}\ll \|f\|_2^2.
\ee

Combining (\ref{111}), (\ref{222}), and (\ref{333}), we get that
\begin{align*}
\int_{\R^d}|\mathcal{M}_{L}^{(j)}(f)(x)|^2\,dx
&\ll \left(L^{1/3}\!\!\!\!\sum_{\ell=\lfloor\log_2\lm_0\rfloor}^\infty2^{-\ell/3}+\Bigl(\frac{L}{\lm_0}\Bigr)^{2/3}\right)\int_{\R^d}|f(x)|^2\,dx\\
&\ll \Bigl(\frac{L}{\lm_0}\Bigr)^{1/3}\int_{\R^d}|f(x)|^2\,dx
\end{align*}
as required.\end{proof}



\end{document}